\documentclass[12pt,leqno]{amsart}
\usepackage{amsfonts,amsthm,amsmath,xcolor,comment}

\theoremstyle{plain}
\newtheorem{thm}{Theorem}[section]

\newtheorem{lem}{Lemma}[section]

\theoremstyle{definition}
\newtheorem{df}{Definition}[section]
\newtheorem{rem}{Remark}[section]
\newtheorem{ex}{Example}[section]

\newcommand{\FF}{\mathbb{F}}
\newcommand{\ZZ}{\mathbb{Z}}

\newcommand{\R}{\mathfrak{R}}

\DeclareMathOperator{\Vsupp}{Vsupp}

\DeclareMathOperator{\wt}{wt}

\begin{document}

\title [On code polynomials II] 
{Weight enumerators, intersection enumerators and Jacobi polynomials II}

\author[Chakraborty] {Himadri Shekhar Chakraborty*}
\address
	{
		Department of Mathematics, 
		Shahjalal University of Science and Techology\\ 
		Sylhet-3114, Bangladesh, 
	}
\email{himadri-mat@sust.edu}

\author[Miezaki] {Tsuyoshi Miezaki}
\address
	{
		Faculty of Science and Engineering\\
		Waseda University\\
		Tokyo 169-8555, Japan
	}
\email{miezaki@waseda.jp}

\author[Oura] {Manabu Oura}
\address
	{
		Faculty of Mathematics and Physics\\
		Kanazawa University\\
		Ishikawa 920--1192, Japan
	}
\email{oura@se.kanazawa-u.ac.jp}

\thanks{*Corresponding author}

\date{}
\maketitle

\begin{abstract}
In the present paper, 
we introduce the concepts of 
Jacobi polynomials and intersection enumerators of codes over 
$\FF_q$ and $\ZZ_{k}$ for arbitrary genus~$g$.
We also discuss the interrelation among them. 
Finally, we give the MacWilliams type identities for Jacobi polynomials. 
\end{abstract}

{\small
\vspace{-0.15in}

\quad
}


\section{Introduction}
Weight enumerators make relationships between 
coding theory, invariant theory, and modular forms 
\cite{{BMS1972},{BMS1999},{BE},{Duke},{Gleason},{MMS1972},{NRS},{Runge}, {SCSS2019}}. 
A striking generalization of the weight enumerators was obtained by
Ozeki \cite{Ozeki}, 
who gave the concept of Jacobi polynomials for codes
in analogy with Jacobi forms~\cite{EZ} of lattices
and presented
a generalization of the MacWilliams identity.
In~\cite{OO}, the notion of the intersection polynomials was given 
for some computations of extremal codes.

In \cite{{HOO}}, the concepts of 
genus $g$ Jacobi polynomials and intersection enumerators of binary codes 
were introduced and the MacWilliams type identities of Jacobi polynomials were given. 
Moreover, the concept of genus $g$ (homogeneous) Jacobi polynomials of binary codes 
in~\cite{HOO} was generalized to  the notion of $g$-fold joint 
(homogeneous) Jacobi polynomials of codes over $\FF_{q}$ and $\ZZ_{k}$ in~\cite{CM2021} 
and the MacWilliams type identity for the generalized notion was presented.
In the present paper, 
we give the generalizations of the concepts discussed 
in~\cite{HOO}, such as genus~$g$ Jacobi polynomials and
genus~$g$ intersection polynomials of binary codes,
to some non-binary cases, like for codes over $\FF_{q}$ and $\ZZ_{k}$.  
Further, we generalize some of the results in~\cite{HOO},
particularly, the MacWilliams type identities for Jacobi polynomials. 

Throughout this paper, 
we assume that $\mathfrak{R}$ denotes 
either the finite field $\FF_{q}$ of order $q$, 
where $q$ is a prime power 
or the ring $\ZZ_{k}$ of integers modulo~$k$ 
for some positive integer $k \ge 2$.
Moreover, $\R^{\ast}$ denotes the set of 
non-zero elements of $\R$. 
We prefer to call the elements of $\R^{n}$ 
as \emph{vectors}.
For $u = (u_{1},\ldots,u_{n}) \in \R^{n}$, 
we denote 
$\wt_{\ell}(u):=\#\{i \mid u_{i} = \ell\}$.

The purpose of this paper is 
to introduce the following polynomials and 
discuss their properties.
Let $[g] := \{1,\ldots,g\}$
and
we denote by
$n_a(u_1, \ldots , u_g )$ 
the number of $i$ such that 
$a = (u_{1,i}, \ldots , u_{g,i})$ 
for $u_{1},\ldots,u_{g} \in \R^{n}$ and $a \in \R^g$. 

\begin{df}\label{Def:All}
Let $g$ be a positive integer and $C$ be an $\R$-linear code of length $n$. 
{Let 
$\binom{[g]}{p} := \{(K_{1},\ldots,K_{p}) \in \ZZ^{p} \mid 1 \leq K_{1} < \cdots < K_{p} \leq g\}$
for any positive integer $p$ such that $1 \leq p \leq g$. 
}
\begin{enumerate}
\item [(1)]
The $g$-th \emph{weight enumerator} of $C$ is
\[
	W^{(g)}
	_C (\{x_a\}_{a\in \R^g}) 
	=
	\sum_{u_1,\ldots,u_g\in C}
	\prod_{a \in \R^g}
	x_a^{n_a(u_1,...,u_g)}.
\]

\item [(2)]
The $g$-th \emph{intersection enumerator} of $C$ is
\begin{align*}
	&I_C^{(g)}
	(\{X_{K,L}\}
	_{1 \leq p \leq g,K\in\binom{[g]}{p},L\in (\R^\ast)^p})\\
	&=
	\sum_{u_1,\ldots,u_g\in C}
	\prod_{1 \leq p \leq g}
	\prod_{\scriptsize
	\begin{array}{c}
		K\in\binom{[g]}{p}
	\end{array}}
	\prod_{L\in {(\R^\ast)^p}}
	X_{K,L}^{n_{L}(u_{K_1},\ldots, u_{K_p})}.
\end{align*}

\item [(3)]
The $g$-th \emph{homogeneous Jacobi polynomial} of $C$ 
with reference vector $v \in \R^{n}$ is
\begin{align*}
	\mathfrak{Jac}^{(g)}_{C,v}(\{y_a\}_{a\in \R^{g+1}})
	=
	\sum_{u_1,\ldots,u_{g}\in C}
	\prod_{a \in \R^{g+1}}
	y_a^{n_a(u_1,...,u_g,v)}.
\end{align*}

\item [(4)]
Assume that $X_{(g+1),(j)}=1$ for all $j\in \R^{\ast}$. 
The $g$-th inhomogeneous Jacobi polynomial of $C$ 
with reference vector $v\in \R^n$ is
\begin{align*}
	&Jac^{(g)}_{C,v}
	(\{X_{K,L}\}
	_{1\leq p\leq g+1,K\in\binom{[g+1]}{p},L\in (\R^\ast)^p})\\
	& =
	\sum_{u_1,\ldots,u_g\in C}
	\prod_{1 \leq p \leq g+1}
	\prod_{\scriptsize
		\begin{array}{c}
			K\in\binom{[g+1]}{p}\\
			u_{g+1}=v
		\end{array}
	}
	\prod_{L\in {(\R^\ast)^p}}
	X_{K,L}^{n_{L}(u_{K_1},\ldots, u_{K_p})}.
\end{align*}
Equivalently, we can also write 
the $g$-th \emph{inhomogeneous Jacobi polynomial} of~$C$ 
with reference vector $v\in \R^n$ as follows:
\begin{align*}
	&Jac^{(g)}_{C,v}
	(\{X_{(k),(\ell)}\}_{k\in[g],j\in\R^{\ast}},
	\{X_{K,L}\}
	_{2 \leq p\leq g+1,K\in\binom{[g+1]}{p},L\in (\R^\ast)^p})\\
	& =
	\sum_{u_1,\ldots,u_g\in C}
	\left(
		\prod_{\scriptsize
			k \in [g]}
		\prod_{\ell \in \R^{\ast}}
		X_{(k),(\ell)}^{n_{\ell}(u_{k})}				
	\right)\\
	&\quad\quad\quad\quad
	\left(
		\prod_{2\leq p\leq g+1}
		\prod_{\scriptsize
			\begin{array}{c}
				K\in\binom{[g+1]}{p}\\
				u_{g+1}=v
			\end{array}}
		\prod_{L\in {(\R^\ast)^p}}
		X_{K,L}^{n_{L}(u_{K_1},\ldots, u_{K_p})}
	\right).
\end{align*}
\end{enumerate}
Note that when we say $u_{1}, \ldots, u_{g} \in C$ in the definitions 
of the four polynomials, 
we mean that the vectors $u_{i}$'s are not necessary to be distinct.
\end{df} 

\begin{ex}\label{Ex:All}
	Let $C_{2}$ be an $\FF_{3}$-linear code with length~$2$ having the codewords: $(0,0), (1,1), (2,2)$.
	The following examples will make the polynomials in Definition~\ref{Def:All} 
	that are complicated looking easier to understand. 
	We derive the polynomials for $g = 2$. 
	Let $v = (1,2) \in \FF_{3}^{2}$ be the reference vector for 
	the homogeneous and inhomogeneous Jacobi polynomials.
	\begin{align*}
		& W_{C_{2}}^{(2)}
		(\{x_a\}_{a\in \FF_{3}^{2}})
		=
		x_{(0,0)}^{2}+
		x_{(1,0)}^{2}+
		x_{(2,0)}^{2}+
		x_{(0,1)}^{2}+
		x_{(1,1)}^{2}+
		x_{(2,1)}^{2}+
		x_{(0,2)}^{2}\\
		& {\hspace{100pt}} +
		x_{(1,2)}^{2}+
		x_{(2,2)}^{2}\\
		& I_{C_{2}}^{(2)}
		(\{X_{K,L}\}
		_{1 \leq p \leq 2,K\in\binom{[2]}{p},L\in (\FF_{3}^\ast)^p})
		=
		1 +
		X_{(1),(1)}^{2} +
		X_{(1),(2)}^{2} +
		X_{(2),(1)}^{2} \\ 
		& \hspace{40pt} +
		X_{(1),(1)}^{2}
		X_{(2),(1)}^{2}
		X_{(1,2),(1,1)}^{2} +
		X_{(1),(2)}^{2}
		X_{(2),(1)}^{2}
		X_{(1,2),(2,1)}^{2} \\
		& \hspace{40pt} +
		X_{(1),(1)}^{2}
		X_{(2),(2)}^{2}
		X_{(1,2),(1,2)}^{2}+
		X_{(1),(2)}^{2}
		X_{(2),(2)}^{2}
		X_{(1,2),(2,2)}^{2}+
		X_{(2),(2)}^{2}	\\
		& \mathfrak{Jac}^{(2)}_{C_{2},v}(\{y_a\}_{a\in \FF_{3}^{3}})
		=
		y_{(0,0,1)}^{1} y_{(0,0,2)}^{1}	+
		y_{(1,0,1)}^{1} y_{(1,0,2)}^{1}	+
		y_{(2,0,1)}^{1} y_{(2,0,2)}^{1}	\\
		& \hspace{90pt} +
		y_{(0,1,1)}^{1} y_{(0,1,2)}^{1}	+
		y_{(1,1,1)}^{1} y_{(1,1,2)}^{1}	+
		y_{(2,1,1)}^{1} y_{(2,1,2)}^{1}	\\
		& \hspace{90pt} +
		y_{(0,2,1)}^{1} y_{(0,2,2)}^{1}	+
		y_{(1,2,1)}^{1} y_{(1,2,2)}^{1}	+
		y_{(2,2,1)}^{1} y_{(2,2,2)}^{1}	\displaybreak\\
		&Jac^{(2)}_{C_2,v}
		(\{X_{K,L}\}
		_{1\leq p\leq 3,K\in\binom{[3]}{p},L\in (\FF_{3}^\ast)^p})
		=
		1 +
		X_{(1),(1)}^{2} 
		X_{(1,3),(1,1)}^{1} X_{(1,3),(1,2)}^{1} \\
		& \hspace{30pt} +
		X_{(1),(2)}^{2} 
		X_{(1,3),(2,1)}^{1} X_{(1,3),(2,2)}^{1} +
		X_{(2),(1)}^{2} 
		X_{(2,3),(1,1)}^{1} X_{(2,3),(1,2)}^{1} \\
		& \hspace{30pt} +
		X_{(1),(1)}^{2} 
		X_{(2),(1)}^{2}
		X_{(1,2),(1,1)}^{2} 
		X_{(1,3),(1,1)}^{1} X_{(1,3),(1,2)}^{1}
		X_{(2,3),(1,1)}^{1} X_{(2,3),(1,2)}^{1}\\
		& \hspace{150pt}
		X_{(1,2,3),(1,1,1)}^{1} X_{(1,2,3),(1,1,2)}^{1}\\
		& \hspace{30pt} +
		X_{(1),(2)}^{2} 
		X_{(2),(1)}^{2}
		X_{(1,2),(2,1)}^{2} 
		X_{(1,3),(2,1)}^{1} X_{(1,3),(2,2)}^{1}
		X_{(2,3),(1,1)}^{1} X_{(2,3),(1,2)}^{1}\\
		& \hspace{150pt}
		X_{(1,2,3),(2,1,1)}^{1} X_{(1,2,3),(2,1,2)}^{1}\\
		& \hspace{30pt} +
		X_{(2),(2)}^{2} 
		X_{(2,3),(2,1)}^{1} X_{(2,3),(2,2)}^{1} +
		X_{(1),(1)}^{2} 
		X_{(2),(2)}^{2}
		X_{(1,2),(1,2)}^{2} 
		X_{(1,3),(1,1)}^{1} \\
		& \hspace{100pt}
		X_{(1,3),(1,2)}^{1}
		X_{(2,3),(2,1)}^{1} X_{(2,3),(2,2)}^{1}
		X_{(1,2,3),(1,2,1)}^{1} X_{(1,2,3),(1,2,2)}^{1}\\
		& \hspace{30pt} +
		X_{(1),(2)}^{2} 
		X_{(2),(2)}^{2}
		X_{(1,2),(2,2)}^{2} 
		X_{(1,3),(2,1)}^{1} X_{(1,3),(2,2)}^{1} 
		X_{(2,3),(2,1)}^{1} X_{(2,3),(2,2)}^{1}\\
		& \hspace{150pt}
		X_{(1,2,3),(2,2,1)}^{1} X_{(1,2,3),(2,2,2)}^{1}
	\end{align*}
\end{ex}

If there is no confusion, 
we prefer to write the notations of the above said 
polynomials in a simple form by omitting the notations
of the variables in the polynomials as 
$W_{C}^{(g)}$, $I_{C}^{(g)}$, $\mathfrak{Jac}^{(g)}_{C,v}$ and $Jac^{(g)}_{C,v}$.

\begin{rem}
	We have the following remarks:
\begin{enumerate}
	\item [(1)]
	It is easy to see that 
	\begin{align*}
		\begin{cases}
		\mathfrak{Jac}^{(g)}_{C,0}=W_C^{(g)}\\
		Jac^{(g)}_{C,0}=I_C^{(g)}. 
		\end{cases}
	\end{align*}
	\item [(2)]
	The number of variables in each polynomial is given by 
	\[
		\left\{
		\begin{array}{lll}
			W_C^{(g)} & : & {|\R|}^g,\\
			I_C^{(g)} & : & \sum_{p=1}^{g}\binom{g}{p}(|\R|-1)^p=|\R|^g-1,\\
			\mathfrak{Jac}^{(g)}_{C,v} & : & |\R|^{g+1},\\
			Jac^{(g)}_{C,v} & : 
			& \sum_{p=1}^{g+1}\binom{g+1}{p}(|\R|-1)^p-(|\R|-1)\\
			& & =|\R|(|\R|^g-1).
		\end{array}
		\right.
	\]
\end{enumerate}
\end{rem}

This paper is organized as follows. 
In Section~\ref{sec:pre}, 
we give definitions and some basic properties of 
codes used in this paper.
In Section \ref{sec:rel}, 
we give relations between the four polynomials 
(Theorem~\ref{Thm:WeightInter}, Theorem~\ref{Thm:HomoInhomo}, Theorem~\ref{Thm:InterJacobi}). 
In Section \ref{sec:mac}, we give the
MacWilliams type identities for a $g$-th homogeneous Jacobi polynomial (Theorem~\ref{Thm:MacHomo}) and a $g$-th inhomogeneous Jacobi polynomial (Theeorem~\ref{Thm:MacInhomo}). 


\section{Preliminaries}\label{sec:pre}

We refer the readers to~\cite{BDHO1999, HP, MS1977} 
for the background of the coding theory. 
Let $\FF_{q}$ be the finite field of order~$q$,
where~$q = p^{f}$ for some prime number~$p$. 
Then $\FF_{q}^{n}$ denotes the $n$-dimensional vector 
space over $\FF_{q}$ equipped with the following inner product
\[
	u \cdot v
	:=
	\sum_{i = 1}^{n}
	u_{i} v_{i},
\]
where
$u = (u_{1},\ldots,u_{n})$,
$v = (v_{1},\ldots,v_{n}) \in \FF_{q}^{n}$.
If $q$ is an even power~$f$ of an arbitrary prime~$p$,
then it is convenient to consider another inner product
given by
\[
	u \cdot v
	:=
	\sum_{i = 1}^{n}
	u_{i}{v_{i}}^{\sqrt{q}}
\]
On the other hand, 
let $\ZZ_{k}$ be the finite ring of integers modulo~$k$
for some positive integer~$k \geq 2$.
Then $\ZZ_{k}^{n}$ denotes the $\ZZ_{k}$-module of 
all~$n$-tuples over~$\ZZ_{k}$.
Let 
$u = (u_{1},\ldots,u_{n})$,
$v = (v_{1},\ldots,v_{n})$
be two elements of~$\ZZ_{k}^{n}$.
Then the inner product of $u$ and $v$ on $\ZZ_{k}^{n}$ 
is defined as
\[
	u \cdot v
	:=
	\sum_{i = 1}^{n}
	u_{i}v_{i}.
\]

An $\R$-\emph{linear code} $C$ of length~$n$ is either a
vector subspace of $\R^{n}$ when $\R$ represents~$\FF_{q}$
or a submodule of $\R^{n}$ if $\R$ denotes~$\ZZ_{k}$.
The \emph{dual} of an $\R$-linear code~$C$
is denoted by $C^{\perp}$ and defined as
\[
	C^{\perp}
	:=
	\{
	v \in \R^{n}
	\mid
	u \cdot v = 0
	\mbox{ for all }
	u \in C
	\}.
\]

{Let $a = (a_{1},\ldots,a_{g}) \in \R^{g}$ with nonzero weight $p$.
Then $\Vsupp(a) := (i \mid a_{i} \neq 0)$,
where $i$'s are in ascending order, is an element of $\binom{[g]}{p}$.
Let $K = (K_{1},\ldots,K_{p}) \in \binom{[g]}{p}$.
Then by $k \in K$ we mean that $k = K_{j}$ for some $1\leq j \leq p$.
Moreover, by $K^{'} \subset K$,
we denote either a $r$-tuple
$K'= (K_{m_{1}},\ldots,K_{m_{r}})$ 
for $1 \leq r \leq p$ 
such that 
$1 \leq m_{1} \leq \cdots \leq m_{r} \leq p$ 
or an empty tuple $K^{'}$
which we prefer to write as~$\emptyset$.
}

Now we have the following useful lemma for this paper.

\begin{lem}\label{lem:1}
Let $u_1, \ldots, u_g$ be elements of $\R^n$. 
Then the following hold. 
\begin{enumerate}
\item [{\rm (1)}]
$\begin{aligned}[t]
	\prod_{\substack{a \in \R^g,\\ {a \neq 0}}}
	\prod_{
		\scriptsize
		\begin{array}c
		K\subset \Vsupp(a),\\
		{K \neq \emptyset}
		\end{array}}
	& X_{K,a_K}^{n_a(u_1,...,u_g)}\\
	& =
	\prod_{1 \leq p\leq g}
	\prod_{\scriptsize
	\begin{array}{c}
	K\in\binom{[g]}{p}
	\end{array}
	}
	\prod_{L\in {(\R^\ast)^p}}
	X_{K,L}^{n_{L}(u_{K_1},\ldots,
	 u_{K_p})},
\end{aligned}\\$
{where $a_K$ is the length $|K|$ vector indexed by $K$ such that 
$a_K=(a_{i_1},\ldots,a_{i_{|K|}})$ with $i_1<\cdots<i_{|K|}$}.

\item[{\rm (2)}]

For $0 \neq a \in \R^{g}$,
$\begin{aligned}[t]
	x_a
	=
	{x_0} 
	\prod_{
		\scriptsize
		\begin{array}c
		K \subset \Vsupp(a),\\
		{K \neq \emptyset}
		\end{array}}
	\prod_{K'\subset K}
	x_{a_{K'}}^{(-1)^{|K|-|K'|}},
\end{aligned}$
{where 
$a_{K'}$
is the length $g$ vector such that 
${a_{K',i}}=a_i$ for $i\in K'$, 
${a_{K',i}}=0$ for $i\not\in K'$, 
and $a_{\emptyset}=0$}.
\end{enumerate}
\end{lem}

\begin{proof}
(1)
The left-hand side counts for $1\leq i\leq n$, the 
number of pairs 
$\{(K,a_K) \mid K\subset \Vsupp(u_{1,i},\ldots,u_{g,i})\}$.
The right-hand side counts for $1\leq i\leq g$, the 
number of pairs 
$\{(K,L) \mid K\subset \Vsupp(u_{i,1},\ldots,u_{i,n})\}$.

(2)
Let the exponent of $x_{a_{K'}}$ in the right-hand  
be $E(x_{a_{K'}})$. 
It is immediate that $x_{a_{K'}}$ appears only once for $K'= \Vsupp(a)$.
Let the length of $\Vsupp(a)$ and $K'$ be $m$ and $\ell$, respectively.
Then for $K' \neq \Vsupp(a)$ with $\ell \neq 0$, we have
\begin{align*}
	E(x_{a_{K'}}) 
	& = 
	\sum_{k = \ell}^{m}
	\#
	\{
		K \in  \binom{[g]}{k}
		\mid 
		K' \subset K \subset \Vsupp(a)\}\\
	& =
	\sum_{k = \ell}^{m}
	(-1)^{k-\ell}
	\binom{m-\ell}{k - \ell}\\
	& =
	\sum_{t = 0}^{m-\ell}
	(-1)^{t}
	\binom{m-\ell}{t}\\
	& =
	(1 + (-1))^{m-\ell}\\
	& = 
	0.
\end{align*} 
Now for $K' \neq \Vsupp(a)$ with $\ell = 0$, we have
\begin{align*}
	E(x_{a_{K'}}) 
	& =
	1 + 
	\sum_{k = 1}^{m}
	\#
	\{
		K \in  \binom{[g]}{k}
		\mid 
		K \subset \Vsupp(a)
	\}\\
	& =
	1 +
	\sum_{k = 1}^{m}
	(-1)^{k}
	\binom{m}{k}\\
	& =
	(1 + (-1))^{m}\\
	& = 
	0.
\end{align*}
Thus the above discussed sums conclude that the right-hand side consists of $x_{a}$ only.
This completes the proof.
\end{proof}

\section{Relations between four polynomials}\label{sec:rel}

\subsection{Weight enumerators and Intersection enumerators}

In this section, 
we give a relation between weight enumerators and intersection enumerators. 

\begin{thm}\label{Thm:WeightInter}
Let $C$ be an $\R$-linear code of length $n$. 
Then the following hold.
\begin{enumerate}
\item [{\rm (1)}]
$\begin{aligned}[t]
	W_C^{(g)}
	\left(
		{x_0\leftarrow 1,
		x_{a}\leftarrow 
		\prod_{
		\scriptsize
		\begin{array}c
		K\subset \Vsupp(a),\\
		K \neq \emptyset
		\end{array}
		}
		X_{K,a_K}
		\mbox{ for } 
		a \neq 0}
	\right)
	=
	I_C^{(g)}, 
\end{aligned}\\$
where 
$a \in \R^{g}$, and $a_K$ is the length $|K|$ vector indexed by $K$ such that 
$a_K=(a_{i_1},\ldots,a_{i_{|K|}})$ with $i_1<\cdots<i_{|K|}$.

\item [{\rm (2)}]
$\begin{aligned}[t]
	x_0^n
	I_C^{(g)}
	\left(
		X_{K,L}\leftarrow 
		\prod_{K'\subset K}
		x_{L_{K'}}^{(-1)^{|K|-|K'|}}
	\right)
	=
	W_C^{(g)}, 
\end{aligned}\\$
where 
$K = (K_{1},\ldots,K_{p})\in \binom{[g]}{p}$
for 
$1 \leq p \leq g$, and
$K' = (K_{m_{1}},\ldots,K_{m_{r}})$ for
$1 \leq r \leq p$ and
$1 \leq m_{1} \leq \cdots \leq m_{r} \leq p$,
and
$L_{K'}$
is the length {$g$} vector such that 
${L_{K',i}}=L_{j}$ for $i = K_{m_{j}} \in K'$ where $1\leq j \leq r$, 
${L_{K',i}}=0$ for $i\not\in K'$,
and $L_{\emptyset}=0$. 
\end{enumerate}
\end{thm}

\begin{proof}
(1)
By Lemma \ref{lem:1} (1), 
\begin{align*}
	&W_C^{(g)}
	\left(
		{x_{0} \leftarrow 1,
		x_{a}\leftarrow 
		\prod_{
			\scriptsize
			\begin{array}c
			K \subset \Vsupp(a)\\
			K \neq \emptyset
			\end{array}}
		X_{K,a_K}
		\mbox{ for }
		a \neq 0}
	\right)\\
	&=
	\sum_{u_1,\ldots,u_g\in C}
	\left(
		\prod_{\substack{a \in \R^g,\\ {a \neq 0}}}
		\prod_{
			\scriptsize
			\begin{array}c
			K\subset \Vsupp(a),\\
			{K \neq \emptyset}
			\end{array}}
		X_{K,a_K}
		^{n_a(u_1,...,u_g)}
	\right)\\
	&=
	\sum_{u_1,\ldots,u_g\in C}
	\prod_{1 \leq p\leq g}
	\prod_{\scriptsize
		\begin{array}{c}
		K\in\binom{[g]}{p}
		\end{array}}
	\prod_{L\in {(\R^{\ast})^p}}
	X_{K,L}^{n_{L}(u_{K_1},\ldots, u_{K_p})}\\
	&=
	I_C^{(g)}. 
\end{align*}
(2)
Let $u_{1},\ldots,u_{g} \in \R^{n}$ and $a \in \R^{g}$. 
We observe for $a \neq 0$ that
$x_{0}^{n} = x_{0}^{n_{0}(u_{1},\ldots,u_{g})} \prod_{a\in\R^{g}} x_{0}^{n_{a}(u_{1},\ldots,u_{g})}$.
Now by Lemma~\ref{lem:1}, 
\begin{align*}
	&x_0^n
	I_C^{(g)}
	\left(
		X_{K,L}
		\leftarrow 
		\prod_{K'\subset K}
		x_{L_{K'}}^{(-1)^{|K|-|K'|}}
	\right)\\
	&=
	x_{0}^{n}
	\sum_{u_1,\ldots,u_g\in C}
	\prod_{1 \leq p \leq g}
	\prod_{\scriptsize
		\begin{array}{c}
			K\in\binom{[g]}{p}
	\end{array}}
	\prod_{L\in {(\R^\ast)^p}}
	\left(
	\prod_{K'\subset K}
	x_{L_{K'}}^{(-1)^{|K|-|K'|}}
	\right)^{n_{L}(u_{K_1},\ldots, u_{K_p})}\\
	&=
	x_{0}^{n}
	\sum_{u_1,\ldots,u_g\in C}
	\prod_{\substack{a \in \R^g,\\ {a \neq 0}}}
	\left(
		\prod_{
		\scriptsize
		\begin{array}c
		K\subset \Vsupp(a),\\
		{K \neq \emptyset}
		\end{array}}
		\prod_{K'\subset K}
		x_{a_{K'}}^{(-1)^{|K|-|K'|}}
	\right)^{n_a(u_1,...,u_g)}\\
	&=
	\sum_{u_1,\ldots,u_g\in C}
	x_0^{n_{0}(u_{1},\ldots,u_{g})}
	\prod_{\substack{a \in \R^g,\\ {a \neq 0}}}
	\left(
		{x_0}
		\prod_{
			\scriptsize
			\begin{array}c
			K\subset \Vsupp(a),\\
			{K \neq \emptyset}
		\end{array}}
		\prod_{K'\subset K}
		x_{a_{K'}}^{(-1)^{|K|-|K'|}}
	\right)^{n_a(u_1,...,u_g)}\\
	& =
	\sum_{u_1,\ldots,u_g\in C}
	x_0^{n_{0}(u_{1},\ldots,u_{g})}
	\prod_{\substack{a \in \R^g,\\{a \neq 0}}}
	x_{a}^{n_{a}(u_{1},\ldots,u_{g})}\\
	& =
	W_C^{(g)}.  
\end{align*}
Hence we complete the proof.
\end{proof}

\begin{ex}\label{Ex:WeightInter}
	Applying Theorem~\ref{Thm:WeightInter} in Example~\ref{Ex:All}, 
	we have 
	\begin{multline*}
		W_{C_{2}}^{(2)}
		(x_{(0,0)} \leftarrow 1, 
		x_{(0,1)} \leftarrow X_{(2),(1)}, 
		x_{(0,2)} \leftarrow X_{(2),(2)},
		x_{(1,0)} \leftarrow X_{(1),(1)},\\
		x_{(1,1)} \leftarrow X_{(1),(1)}X_{(2),(1)}X_{(1,2),(1,1)},
		x_{(1,2)} \leftarrow X_{(1),(1)}X_{(2),(2)}X_{(1,2),(1,2)},\\
		x_{(2,0)} \leftarrow X_{(1),(2)},
		x_{(2,1)} \leftarrow X_{(1),(2)}X_{(2),(1)}X_{(1,2),(2,1)},\\
		x_{(2,2)} \leftarrow X_{(1),(2)}X_{(2),(2)}X_{(1,2),(2,2)})	\\
		=
		I_{C_{2}}^{(2)}
		(X_{(1),(1)},X_{(1),(2)},
		X_{(2),(1)},X_{(2),(2)},\\
		X_{(1,2),(1,1)},X_{(1,2),(1,2)},X_{(1,2),(2,1)},X_{(1,2),(2,2)}).	
	\end{multline*}

	\begin{multline*}
		x_{(0,0)}^{2}
		I_{C_{2}}^{(2)}
		(X_{(1),(1)} \leftarrow \frac{x_{(1,0)}}{x_{(0,0)}},
		X_{(1),(2)} \leftarrow \frac{x_{(2,0)}}{x_{(0,0)}},
		X_{(2),(1)} \leftarrow \frac{x_{(0,1)}}{x_{(0,0)}},\\
		X_{(2),(2)} \leftarrow \frac{x_{(0,2)}}{x_{(0,0)}}, 
		X_{(1,2),(1,1)} \leftarrow \frac{x_{(0,0)}x_{(1,1)}}{x_{(1,0)}x_{(0,1)}},
		X_{(1,2),(1,2)} \leftarrow \frac{x_{(0,0)}x_{(1,2)}}{x_{(1,0)}x_{(0,2)}},\\
		X_{(1,2),(2,1)} \leftarrow \frac{x_{(0,0)}x_{(2,1)}}{x_{(2,0)}x_{(0,1)}},
		X_{(1,2),(2,2)} \leftarrow \frac{x_{(0,0)}x_{(2,2)}}{x_{(2,0)}x_{(0,2)}})\\
		=
		W_{C_{2}}^{(2)}(x_{(0,0)},x_{(0,1)},x_{(0,2)},x_{(1,0)},x_{(1,1)},x_{(1,2)},
		x_{(2,0)},x_{(2,1)},x_{(2,2)}).
	\end{multline*}
\end{ex}


\subsection{Homogeneous and inhomogeneous Jacobi polynomials}
In this section, 
we give a relation between homogeneous and inhomogeneous Jacobi polynomials. 

\begin{thm}\label{Thm:HomoInhomo}
Let $C$ be an $\R$-linear code of length $n$. Then
\begin{align*}
	\mathfrak{Jac}^{(g)}_{C,v}&(\{y_a\}_{a\in \R^{g+1}})\\
	&=
	(y_0)^n
	\prod_{\ell\in \R^{\ast}}
	\left(\frac{y_{(0,\ldots,0,\ell)}}{y_0}\right)^{\wt_\ell(v)}\\
	&\quad\quad
	{Jac}^{(g)}_{C,v}
	(X_{(k),(\ell)} \leftarrow
	\left\{\dfrac{y_{L_{(k)}}}{y_{0}}\right\},
	X_{K,L}\leftarrow 
	\prod_
	{
	K'\subset K
	}
	y_{L_{K'}}^{(-1)^{|K|-|K'|}}
	), 
\end{align*}
where 
$L_{(k)}$ denotes a vector with length~$g+1$ such that 
$L_{(k),i} =\ell$ if $i = k$, $L_{(k),i} = 0$ if $i \neq k$, and
$K = (K_{1},\ldots,K_{p}) \in \binom{[g+1]}{p}$ 
for $2 \leq p \leq g+1$, and
$K' = (K_{m_{1}},\ldots,K_{m_{r}})$ for
$1 \leq r \leq p$ and
$1 \leq m_{1} \leq \cdots \leq m_{r} \leq p$,
and
$L_{K'}$ is the length {$g+1$} vector such that 
${L_{K',i}}=L_{j}$ for $i = K_{m_{j}} \in K'$, 
where $1\leq j \leq r$, 
${L_{K',i}}=0$ for $i \notin K'$,
 and $L_{\emptyset}=0$. 
\end{thm}

\begin{proof}
	From the right-hand side and using Lemma~\ref{lem:1}, 
	we have
	\begin{align*}
		& (y_0)^n
		\prod_{\ell\in \R^\ast}
		\left(\frac{y_{(0,\ldots,0,\ell)}}{y_0}\right)^{\wt_\ell(v)}
		{Jac}^{(g)}_{C,v}
		(X_{(k),(\ell)} \leftarrow
		\left\{\dfrac{y_{L_{(k)}}}{y_{0}}\right\},\\
		&\hspace{7cm}
		X_{K,L}\leftarrow 
		\prod_
		{
			K'\subset K
		}
		y_{L_{K'}}^{(-1)^{|K|-|K'|}}
		)\\
		& =
		(y_0)^n
		\sum_{u_1,\ldots,u_g\in C}
		\left(
		\prod_{\scriptsize
			k \in [g]}
		\prod_{\ell \in \R^{\ast}}
		\left\{\dfrac{y_{L_{(k)}}}{y_{0}}\right\}^{n_{\ell}(u_{k})}				
		\right)\\
		& \quad\quad
		\prod_{2 \leq p\leq g+1}
		\prod_{\scriptsize
			\begin{array}{c}
				K\in\binom{[g+1]}{p},\\
				u_{g+1}=v
			\end{array}
		}
		\prod_{L\in {(\R^\ast)^p}}
		\left(
			\prod_{K'\subset K}
			y_{L_{K'}}^{(-1)^{|K|-|K'|}}
		\right)^{n_{L}(u_{K_1},\ldots, u_{K_p})}\\
		& =
		\sum_{u_1,\ldots,u_g\in C}
		y_{0}^{n_{0}(u_{1},\ldots,u_{g},v)}
		\prod_{\substack{a \in \R^{g+1},\\ a \neq 0}}
		\left(
			y_0
			\prod_{
				\scriptsize
				\begin{array}c
					K\subset \Vsupp(a),\\
					K \neq \emptyset,\\
					v = u_{g+1}
				\end{array}}
			\prod_{K'\subset K}
			y_{a_{K^{\prime}}}^{(-1)^{|K|-|K^{\prime}|}}
		\right)^{n_a(u_1,...,u_g,v)}\\
		& =	
		\sum_{u_1,\ldots,u_g\in C}
		y_{0}^{n_{0}(u_{1},\ldots,u_{g},v)}
		\prod_{\substack{a \in \R^{g+1},\\ a \neq 0}}
		y_{a}^{n_{a}(u_{1},\ldots,u_{g},v)}\\
		& =
		\mathfrak{Jac}^{(g)}_{C,v}
		(\{y_{a}\}_{a \in \R^{g+1}}).
	\end{align*}
This completes the proof.
\end{proof}

\subsection{Intersection enumerators and Jacobi polynomials}

In this section, 
we give a relation between intersection enumerators and Jacobi polynomials. 

\begin{thm}\label{Thm:InterJacobi}
Let $C$ be an $\R$-linear code of length $n$. 
{Let 
$s := |\R^{\ast}| $. 
Denote the elements of 
$\R^{\ast}$ by  $\ell_{1},\ldots,\ell_{s}$}.
Then the following hold.
\begin{align*}
I_C^{(g+1)}
=
\sum_{r=0}^{n}
\sum_{\scriptsize
\begin{array}{c}
r_1,\ldots,r_s\in \ZZ_{\geq 0}\\
r_1+\cdots +r_s=r\\
\ell_1,\ldots, \ell_s\in\R^{{\ast}}
\end{array}
}
\left(
	\sum_{\substack{v\in C,\\ \wt_{\ell_i}(v)=r_i}} 
	Jac_{C,v}^{(g)}
\right)
X_{(g+1),\ell_{1}}^{r_1}
\cdots
X_{(g+1),\ell_{s}}^{r_s}.
\end{align*}
\end{thm}

\begin{proof}
	Let $C$ be an $\R$-linear code of length~$n$. 
	Then the $(g+1)$-th intersection enumerator for a code~$C$ 
	can be written as:
	\begin{align*}
		I_C^{(g+1)}
		&(\{X_{K,L}\}
			_{1 \leq p\leq g+1,K\in\binom{[g+1]}{p},L\in (\R^\ast)^p})\\
		&=
		\sum_{u_1,\ldots,u_{g+1}\in C}
		\prod_{1 \leq p\leq g+1}
		\prod_{\scriptsize
			\begin{array}{c}
				K\in\binom{[g+1]}{p}
			\end{array}
		}
		\prod_{L\in {(\R^\ast)^p}}
		X_{K,L}^{n_{L}(u_{K_1},\ldots, u_{K_p})}\\
		& =
		\sum_{u_{g+1}\in C}
		\left\{
			\sum_{u_1,\ldots,u_{g}\in C}
			\left(
				\prod_{\scriptsize
					k \in [g]}
				\prod_{\ell \in \R^{\ast}}
				X_{(k),(\ell)}^{n_{\ell}(u_{k})}				
			\right)\right.\\
			&\hspace{1cm}
			\left.
			\prod_{2 \leq p\leq g+1}
			\prod_{\scriptsize
				\begin{array}{c}
					K\in\binom{[g+1]}{p}
				\end{array}
			}
			\prod_{L\in {(\R^\ast)^p}}
			X_{K,L}^{n_{L}(u_{K_1},\ldots, u_{K_p})}
		\right\}
		\prod_{L \in \R^{\ast}}
		X_{(g+1),L}^{n_{L}(u_{g+1})}\\
		& =
		\sum_{u_{g+1}\in C}
		Jac^{(g)}_{C,u_{g+1}} 
		(\{X_{(k),(\ell)}\}_{k \in [g], \ell \in \R^\ast},\\
		&\hspace{3cm}
		\{X_{K,L}\}
		_{2 \leq p\leq g+1,K\in\binom{[g+1]}{p},L\in (\R^\ast)^p})	
		\prod_{L \in \R^{\ast}}
		X_{(g+1),L}^{n_{L}(u_{g+1})}\\
		& =
		\sum_{r = 0}^{n}
		\sum_{
			\substack{r_{1},\ldots,r_{s} \in \ZZ_{\geq 0}\\
						r_{1}+\cdots+r_{s} = r\\
						\ell_{1},\ldots,\ell_{s} \in \R^{\ast}}}
			\sum_{\substack{v \in C,\\ \wt_{\ell_i}(v) = r_{i}}}
			Jac^{(g)}_{C,v}
			(\{X_{(k),(\ell)}\}_{k \in [g], \ell \in \R^{\ast}},\\
			& \hspace{2cm}
			\{X_{K,L}\}
			_{2 \leq p\leq g+1,K\in\binom{[g+1]}{p},L\in (\R^\ast)^p})
		X_{(g+1),\ell_{1}}^{r_{1}} \cdots X_{(g+1),\ell_{s}}^{r_{s}}.
	\end{align*}
	Hence the proof is completed. 
\end{proof}

\section{MacWilliams type identities}\label{sec:mac}
In this section, 
we give two MacWilliams type identities for 
$g$-th homogeneous and $g$-th inhomogeneous Jacobi polynomials. 
We recall~\cite{DHO,MMS1972}
to introduce some fixed characters over~$\R$.
A \emph{character} $\chi$ of $\R$ 
is a homomorphism from the additive group $\R$ 
to the multiplicative group of non-zero complex numbers.

Let $\R = \FF_{q}$, 
where $q =  p^{f}$ for some prime number $p$.  
Now let $F(x)$ be a primitive irreducible polynomial 
of degree $f$ over $\FF_{p}$ 
and let $\lambda$ be a root of $F(x)$.
Then any element $\alpha \in \FF_{q}$ 
has a unique representation as:
\begin{equation}\label{EquAlpha}
	\alpha 
	=
	\alpha_{0} + \alpha_{1} \lambda	
	+ \alpha_{2} \lambda^{2}
	+ \cdots +
	\alpha_{f-1} \lambda^{f-1},
\end{equation} 
where $\alpha_{i} \in \FF_{p}$.
We define $\chi(\alpha) := \zeta_{p}^{\alpha_{0}}$, 
where $\zeta_{p}$ is the primitive complex $p$-th root of unity 
$\zeta_{p} = e^{2{\pi}i/p}$,
and $\alpha_{0}$ is given by Equation~(\ref{EquAlpha}).

Again if $\R = \ZZ_{k}$,
then for
$\alpha\in\ZZ_{k}$,
we define 
$\chi$ as $\chi(\alpha) := \zeta_{k}^{\alpha}$,
where
$\zeta_{k}$ is the primitive complex $k$-th root of unity
$\zeta_{k} = e^{2{\pi}i/k}$.

For any $\alpha \in \R$, 
we have the following property:
\[
	\sum_{\beta \in \R}
	\chi(\alpha\beta) 
	:= 
	\begin{cases}
		|\R| & \mbox{if} \quad \alpha = 0, \\
		0 & \mbox{if} \quad \alpha \neq 0.
	\end{cases} 
\]


\subsection{Homogeneous Jacobi polynomials}

The MacWilliams type identity for the $g$-fold
complete joint Jacobi polynomials were discussed 
in~\cite[Theorem 5.1]{CM2021}. Then it is easy to 
give the MacWilliams type identity for 
the $g$-th homogeneous Jacobi polynomials as follows.

\begin{thm}\label{Thm:MacHomo}
Let 
$C$ be an $\R$-linear code of length $n$.  
Then we have
\begin{align*}
&\mathfrak{Jac}^{(g)}_{C^{\perp}, v}(\{y_a\}_{a\in \R^{g+1}})
\\
&=
\frac{1}{|C|^g}
\mathfrak{Jac}^{(g)}_{C, v}
\left(\left\{
\sum_{
\scriptsize
\begin{array}{c}
b_1,\ldots,b_g\in \R\\
b_{g+1}=a_{g+1}
\end{array}
}
\chi
\left(
	\sum_{i = 1}^{g} a_{i}b_{i}
\right)
y_{(b_1,\ldots,b_g,b_{g+1})}
\right\}
_{a\in \R^{g+1}}\right)_{\,.}
\end{align*}
\end{thm}
\begin{proof}
The proof is similar to the proof of \cite[Theorem 12]{HOO}.
\end{proof}

\subsection{Inhomogeneous Jacobi polynomials}
In this section, 
we give a MacWilliams type identity for 
inhomogeneous Jacobi polynomials. 

\begin{thm}\label{Thm:MacInhomo}
Let $C$ be an $\R$-linear code of length $n$. 
Then we have the following MacWilliams type relation:
\begin{align*}
&Jac^{(g)}_{C^\perp,v}
(\{X_{(k),(\ell)}\}_{k \in [g], \ell \in \R^{\ast}},
\{X_{K,L}\}
_{2 \leq p\leq g+1,K\in\binom{[g+1]}{p},L\in (\R^\ast)^p})\\
&=
(y_0)^n
\prod_{\ell\in \R^\ast}
\left(\frac{y_{(0,\ldots,0,\ell)}}{y_0}\right)^{\wt_\ell(v)}\\
&\quad\quad
{Jac}^{(g)}_{C,v}
(X_{(k),(\ell)} \leftarrow
\left\{\dfrac{y_{L_{(k)}}}{y_{0}}\right\}, 
X_{K,L}\leftarrow 
\prod_
{
K'\subset K
}
y_{L_{K'}}^{(-1)^{|K|-|K'|}}
), 
\end{align*}
where 
$L_{(k)}$ denotes a vector with length~$g+1$ such that 
$L_{(k),i} =\ell$ if $i = k$, $L_{(k),i} = 0$ if $i \neq k$, and
$K = (K_{1},\ldots,K_{p}) \in \binom{[g+1]}{p}$ 
for $2 \leq p \leq g+1$, and
$K' = (K_{m_{1}},\ldots,K_{m_{r}})$ for
$1 \leq r \leq p$ and
$1 \leq m_{1} \leq \cdots \leq m_{r} \leq p$,
and
$L_{K'}$ is the length {$g+1$} vector such that 
${L_{K',i}}=L_{j}$ for $i = K_{m_{j}} \in K'$, where $1\leq j \leq r$, 
${L_{K',i}}=0$ for $i \notin K'$,
and $L_{\emptyset}=0$,
and for $a\in \R^{g+1}$, 
\begin{align*}
&y_a=
\sum_{\scriptsize
\begin{array}{c}
b_1,\ldots,b_g\in \R\\
b_{g+1}=a_{g+1}
\end{array}
}
{\chi(\sum_{i=1}^g a_ib_i)} 
\prod_{1 \leq p\leq g+1}
\prod_{\scriptsize
\begin{array}{c}
K\in\binom{[g+1]}{p}
\end{array}
}
X_{K,B_K}, 
\end{align*}
where
$B_K$ is the length $|K|$ vector indexed by $K$ such that 
$B_K=(b_{i_1},\ldots,b_{i_{|K|}})$ with $i_1<\cdots<i_{|K|}$,
and 
$X_{\emptyset,B_{\emptyset}}=1$, and 
$X_{(g+1),(\ell)}=1$ for all $\ell\in \R$. 

\end{thm}

\begin{proof}
	For any $w \in \R^{n}$, we define
	\[
		\delta_{C^{\perp}}(w)
		:=
		\begin{cases}
			1 & \mbox{if } w \in C^{\perp},\\
			0 & \mbox{otherwise}.
		\end{cases}
	\]
	Then we have the following identity
	\[
		\delta_{C^{\perp}}(w)
		=
		\dfrac{1}{|C|}
		\sum_{u \in C}
		\chi(u \cdot w).
	\]
	Now
	\begin{align*}
		Jac^{(g)}_{C^{\perp},v}&
		(\{X_{(k),(\ell)}\}_{k\in [g], \ell\in\R^{\ast}}, 
		\{X_{K,L}\}_{
			2 \leq p \leq g+1, 
			K \in \binom{[g+1]}{p}, 
			L \in (\R^{\ast})^{p}})\\
		& =
		\sum_{u_1,\ldots,u_g\in C^{\perp}}
		\left(
		\prod_{\scriptsize
			k \in [g]}
		\prod_{\ell \in \R^{\ast}}
		X_{(k),(\ell)}^{n_{\ell}(u_{k})}				
		\right)\\
		&\hspace{2cm}
		\left(
		\prod_{2\leq p\leq g+1}
		\prod_{\scriptsize
			\begin{array}{c}
				K\in\binom{[g+1]}{p}\\
				u_{g+1}=v
		\end{array}}
		\prod_{L\in {(\R^\ast)^p}}
		X_{K,L}^{n_{L}(u_{K_1},\ldots, u_{K_p})}
		\right)\\
		& =
		\sum_{u_1,\ldots,u_g\in C^{\perp}}
		\prod_{1 \leq p \leq g+1}
		\prod_{\scriptsize
			\begin{array}{c}
				K\in\binom{[g+1]}{p}\\
				K \neq (g+1)\\
				u_{g+1}=v
			\end{array}
		}
		\prod_{L\in {(\R^\ast)^p}}
		X_{K,L}^{n_{L}(u_{K_1},\ldots, u_{K_p})}\\
		& =
		\sum_{w_1,\ldots,w_g\in \R^{n}}
		\prod_{i = 1}^{g}
		\delta_{C^{\perp}}(w_{i})
		\prod_{1 \leq p\leq g+1}
		\prod_{\scriptsize
			\begin{array}{c}
				K\in\binom{[g+1]}{p}\\
				K \neq (g+1)\\
				w_{g+1}=v
			\end{array}
		}
		\prod_{L\in {(\R^\ast)^p}}
		X_{K,L}^{n_{L}(w_{K_1},\ldots, w_{K_p})}\\
		& =
		\sum_{w_1,\ldots,w_g\in \R^{n}}
		\prod_{i = 1}^{g}
		\left(
			\dfrac{1}{|C|}
			\sum_{u_{i} \in C}
			\chi(u_{i} \cdot w_{i})
		\right)\\
		& \hspace{60pt}
		\prod_{1 \leq p\leq g+1}
		\prod_{\scriptsize
			\begin{array}{c}
				K\in\binom{[g+1]}{p}\\
				K \neq (g+1)\\
				w_{g+1}=v
			\end{array}
		}
		\prod_{L\in {(\R^\ast)^p}}
		X_{K,L}^{n_{L}(w_{K_1},\ldots, w_{K_p})}\displaybreak\\
		& =
		\dfrac{1}{|C|^{g}}
		\sum_{
			\substack{
				u_{1},\ldots,u_{g} \in C\\
				w_1,\ldots,w_g\in \R^{n}\\
				w_{g+1} = v}}
		\chi(\sum_{i = 1}^{g}u_{i,1}w_{i,1}+\cdots+u_{i,n}w_{i,n})\\
		& \hspace{60pt}
		\prod_{1 \leq p\leq g+1}
		\prod_{\scriptsize
			\begin{array}{c}
				K\in\binom{[g+1]}{p}\\
				K \neq (g+1)
			\end{array}}
		\prod_{L\in {(\R^\ast)^p}}
		X_{K,L}^{n_{L}(w_{K_1},\ldots, w_{K_p})}\\
		& =
		\dfrac{1}{|C|^{g}}
		\sum_{
			\substack{
				u_{1},\ldots,u_{g} \in C}}
		\prod_{1 \leq i \leq n}
		\left\{
			\sum_{w_{1,i},\ldots,w_{g,i}\in\R}
			\chi(u_{1,i}w_{1,i}+\cdots+u_{g,i}w_{g,i})
		\right.\\
		& 
		\left.
		\left(
			\prod_{1 \leq p\leq g+1}
			\prod_{\scriptsize
				\begin{array}{c}
					K\in\binom{[g+1]}{p}\\
					K \neq (g+1)\\
					w_{g+1}=v
			\end{array}}
			X_{K,(w_{K_1,i},\ldots,w_{K_p,i})}
		\right)
		\right\}\\
		& =
		\dfrac{1}{|C|^{g}}
		\sum_{
			\substack{
				u_{1},\ldots,u_{g} \in C}}
		\prod_{a \in \R^{g+1}}
		\left\{
		\sum_{
			\substack{b_{1},\ldots,b_{g}\in\R\\
				b_{g+1} = a_{g+1}}}
		\chi(a_{1}b_{1} + \cdots + a_{g}b_{g})
		\right.\\
		& \hspace{100pt}
		\left.
		\left(
		\prod_{1 \leq p\leq g+1}
		\prod_{\scriptsize
			\begin{array}{c}
				K\in\binom{[g+1]}{p}\\
				K \neq (g+1)
		\end{array}}
		X_{K,B_{K}}
		\right)
		\right\}^{n_{a}(u_{1},\ldots,u_{g},v)}\\
		& =
		\dfrac{1}{|C|^{g}}
		\sum_{
			\substack{
				u_{1},\ldots,u_{g} \in C}}
		\prod_{a \in \R^{g+1}}
		y_{a}^{{n_{a}(u_{1},\ldots,u_{g},v)}}\\
		& =
		\dfrac{1}{|C|^{g}}
		\sum_{
			\substack{
				u_{1},\ldots,u_{g} \in C}}
		y_{0}^{n}\\
		&\hspace{1cm}
		\prod_{1 \leq p \leq g+1}
		\prod_{
			\scriptsize
			\begin{array}{c}
				K \in \binom{[g+1]}{p}\\ 
				v = u_{g+1}
			\end{array}}
		\prod_{L \in (\R^{\ast})^{p}}
		\left\{
			\prod_{K'\subset K}
			y_{L_{K'}}^{(-1)^{|K|-|K'|}}
		\right\}^{n_{L}(u_{K_{1}},\ldots,u_{K_{p}})}\\
		&\hspace{8cm}
		(\mbox{by Lemma}~\ref{lem:1})\displaybreak\\
		& =
		\dfrac{1}{|C|^{g}}
		y_{0}^{n}
		\prod_{\ell \in \R^{\ast}}
		\left(
			\dfrac{y_{(0,\ldots,0,\ell)}}{y_{0}}
		\right)^{\wt_\ell(v)}\\
		& \hspace{15pt}
		\sum_{
			\substack{
				u_{1},\ldots,u_{g} \in C}}
		\prod_{1 \leq p \leq g+1}
		\prod_{
			\scriptsize
			\begin{array}{c}
				K \in \binom{[g+1]}{p}\\ 
				K \neq (g+1)\\
				v = u_{g+1}
		\end{array}}
		\prod_{L \in (\R^{\ast})^{p}}
		\left\{
		\prod_{K'\subset K}
		y_{L_{K'}}^{(-1)^{|K|-|K'|}}
		\right\}^{n_{L}(u_{K_{1}},\ldots,u_{K_{p}})}.
	\end{align*}
	This completes the proof.
\end{proof}

We will discuss a generalization of the Brou\'e--Enguehard map~\cite{BE} 
and Bannai-Ozeki map~\cite{BO} to the case Jacobi polynomials of genus~$g$
in our subsequent papers.
The notion of Jacobi polynomials with multiple reference
vectors would be discussed in~\cite{CMTOxxxx}.
Bonnecaze et al.~\cite{BMS1999} gave a connection between design theory and Jacobi polynomials. Moreover,
a generalization of the connection 
to the case genus~$g$ Jacobi polynomials with multiple reference vectors would be given in~\cite{CHMOxxxx}.
This gives rise to a natural question:
is it possible to generalize the results of this paper to the case genus~$g$ Jacobi polynomials with multiple reference vectors $g$?
In our future work, we shall investigate this question.

\section*{Acknowledgements}
This work was supported by JSPS KAKENHI (18K03217).
The authors would also like to thank the anonymous
referees for their beneficial comments on an earlier version of the manuscript.

\end{document}